\documentclass[12pt]{article}
\usepackage{amsmath}
\usepackage{latexsym}
\usepackage{amssymb}
\usepackage{graphicx}
\usepackage[numbers,sort&compress]{natbib}
\usepackage{arydshln}

\newcommand{\E}{{\cal E}}

\newcommand{\F}{{\cal F}}

\newtheorem{theorem}{Theorem}[section]
\newtheorem{definition}{Definition}[section]
\newtheorem{lemma}[theorem]{Lemma}
\newtheorem{conjecture}[theorem]{Conjecture}

\newtheorem{construction}[theorem]{Construction}

\def\whitebox{{\hbox{\hskip 1pt
 \vrule height 6pt depth 1.5pt
 \lower 1.5pt\vbox to 7.5pt{\hrule width
    3.2pt\vfill\hrule width 3.2pt}%
 \vrule height 6pt depth 1.5pt
 \hskip 1pt } }}
\def\qed{\ifhmode\allowbreak\else\nobreak\fi\hfill\quad\nobreak
     \whitebox\medbreak}

\newcommand{\ignore}[1]{}
\begin{document}
\begin{sloppypar}
\title{ Constructions for supersaturation of eventown problems}

\author
{\small \  Xiaolei Niu$^{1}$, Yinghui Hang$^{2}$, and Haitao Cao$^{2}$\\
\small   \scriptsize 1     School of Mathematical Sciences, Jiangsu Second Normal University, Nanjing 210013, China\\
\small   \scriptsize 2   School of Mathematical Sciences, Ministry of Education Key Laboratory for NSLSCS,\\ \scriptsize Nanjing Normal University, Nanjing 210023, China.
}
\date{}
\maketitle

\begin{abstract}
In this paper, we study the supersaturation problems of eventown.
Given a family $\mathcal{A}$ of subsets of an $n$ element set, let op$(\mathcal{A})$ denote the number of distinct pairs $A,B\in \mathcal{A}$ for which $|A\cap B|$ is odd. We give extremal eventown constructions and show that for fixed $s\le2^{\lfloor \frac{n}{2} \rfloor}-2$, there exists a collection of $2^{\lfloor\frac{n}{2}\rfloor}+s$ even-sized subsets of an $n$ element set that contains exactly $s\cdot 2^{\lfloor \frac{n}{2} \rfloor-1}$ pairwise intersections of odd size. This extends the range of $s$ in a conjecture proposed by O'Neill from $2^{\lfloor \frac{n}{2} \rfloor}-2^{\lfloor \frac{n}{4} \rfloor}$ to $2^{\lfloor \frac{n}{2} \rfloor}-2$. We also give a construction using symmetric designs to prove that when $k$ is even and $4k-1$ is a prime power, there exists a collection of $2^{\lfloor\frac{4k-1}{2}\rfloor}+s$ even-sized subsets of a $4k-1$ element  set $\mathcal{A}_s$ with  $op(\mathcal{A}_s)=s \cdot 2^{{\lfloor\frac{4k-1}{2}\rfloor}-1}$,  $1\leq s\leq4k-1$.

\bigskip

\end{abstract}

\section{Introduction}

Let $[n]:=\{1,2,\cdots, n\}$, $2^{[n]}$ denote the collection of all subsets of $[n]$,  $\binom{[n]}{k}$ denote the collection of all $k$-subsets of $[n]$, and $\bar{A}$ denote $[n]\setminus A$ where $A$ is a subset of $[n]$. In extremal set theory, given a finite family (i.e., a collection of subsets) $\mathcal{F}$ and a restriction on the intersection of two subsets, the restricted intersection problem asks for the maximum size of a subfamily $\mathcal{A}\subset \mathcal{F}$ such that any two different members of $\mathcal{A}$ satisfy the restricted intersection. Over the past few decades, several restricted intersection problems have been well researched, such as L-intersecting families and bounded symmetric differences, one can refer to \cite{MR0140419,MR0002086,MR3656338,MR3838458,MR0200179,MR1343352,MR1544925}.

Let $\mathcal{A}=\{A_1,A_2,\cdots,A_m\}$ be a family of subsets of $[n]$. We say $\mathcal{A}$ is an {\it eventown (oddtown) family} if all its sets have even(odd) size and $|A_i\cap A_j|$   is  even  for $1\le i\textless j\le m$.
Berlekamp \cite{MR0249303} and Graver \cite{MR0376400} determined independently that the maximum size of an eventown(resp. oddtown) family  is $2^{\lfloor \frac{n}{2} \rfloor}$(resp. $n$). Their results relied on linear algebra which has been widely used in extremal combinatorics. Numerous extensions and variants of the eventown and oddtown problems can be found in the literature \cite{MR0729786,MR0725068,JO2023,MR4429337,MR3755659,MR2191021,MR1451492,MR1773658,Wei2025DCC}.

In this paper, we study the supersaturation versions. O'Neill \cite{MR4564468} first initiated the study of supersaturation problem: if $\mathcal{A}\subset 2^{[n]}$ is a family of more than $2^{\lfloor \frac{n}{2} \rfloor}$ even-sized subsets(resp. $n$ odd-sized subsets), how many pairs of members in $\mathcal{A}$ would violate the intersecting restriction, i.e. the size of the intersection is odd? The \emph{odd pair number} for a given set family $\mathcal{A}$, denoted as  $op(\mathcal{A})$, is the number of pairs of distinct members $A,B\in \mathcal{A}$ such that $|A\cap B|$ is odd.

For supersaturation of eventown, O'Neill in \cite{MR4564468} constructed a family of even-sized subsets with size $2^{\frac{n}{2}}+s$ whose odd pair number is $s\cdot 2^{\frac{n}{2}-1}$ when $n$ is doubly even. He proved that is best possible when $s=1$ and $s=2$ by the linear algebra method.  Furthermore, he proposed the following conjectures.

\begin{conjecture}\label{COA}(\cite{MR4564468})
 Let $n\textgreater 1$ and fix $1\le s\le 2^{\lfloor \frac{n}{2} \rfloor}-2^{\lfloor \frac{n}{4} \rfloor}$. If $\mathcal{A}\subset 2^{[n]}$ consists of even-sized subsets with $|\mathcal{A}|\ge 2^{\lfloor \frac{n}{2} \rfloor}+s$, then op$(\mathcal{A})\ge s\cdot 2^{\lfloor \frac{n}{2} \rfloor-1}$.
\end{conjecture}

Recently, some progress was made in \cite{Wei2023OnSF}. For Conjecture \ref{COA}, the lower bound for sufficiently large $n$ and some range of $s$ was proved true by extremal graph theory.

\begin{theorem}\label{RA}(\cite{Wei2023OnSF})
Let $n$ be a large enough integer and fix $s\in[2^{\lfloor \frac{n}{8} \rfloor}/n]$. Any  family of even-sized subsets $\mathcal{A}\subset 2^{[n]}$  with $|\mathcal{A}|\ge 2^{\lfloor \frac{n}{2} \rfloor}+s$ satisfies $op(\mathcal{A})\ge s\cdot 2^{\lfloor \frac{n}{2} \rfloor-1}$.
\end{theorem}

Our main result is that we construct a family of even-sized subsets $\mathcal{A}$ with $|\mathcal{A}|= 2^{\lfloor \frac{n}{2} \rfloor}+s$ and op$(\mathcal{A})= s\cdot 2^{\lfloor \frac{n}{2} \rfloor-1}$, $1\le s\le 2^{\lfloor \frac{n}{2} \rfloor}-2$. We also give a construction using symmetric designs to prove that when $k$ is even and $4k-1$ is a prime power, there exists a collection of $2^{\lfloor\frac{4k-1}{2}\rfloor}+s$ even-sized subsets of a $4k-1$ element  set $\mathcal{A}_s$ with  $op(\mathcal{A}_s)=s \cdot 2^{{\lfloor\frac{4k-1}{2}\rfloor}-1}$,  $1\leq s\leq4k-1$.

\begin{theorem}\label{Th}
Let $n$, $s$ be positive integers, and fix $1\leq s \leq 2^{\lfloor\frac{n}{2}\rfloor}-2$, there is a family of even-sized subsets $\mathcal{A}\subset 2^{\left[n\right]}$ with $|\mathcal{A}|= 2^{\lfloor \frac{n}{2} \rfloor}+s$ and $op(\mathcal{A})= s\cdot 2^{\lfloor \frac{n}{2} \rfloor-1} $.
\end{theorem}

 \begin{theorem}\label{SD}
  When $k$ is even and $4k-1$ is a prime power, there exists a family of even-sized subsets $\mathcal{A}_s$ with $|\mathcal{A}_s|=2^{\lfloor\frac{4k-1}{2}\rfloor}+s$, $1\leq s\leq4k-1$, and $op(\mathcal{A}_s)=s \cdot 2^{{\lfloor\frac{4k-1}{2}\rfloor}-1}$.
 \end{theorem}

\section{Supersaturation of eventown}

In this section, we focus on  supersaturation problems of eventown. Both of our constructions are based on  $1$-factors of $K_n$ with vertex set $[n]$.  A factor of a graph $G$ is a subgraph of $G$ including all vertices of $G$ (a spanning subgraph). A $1$-factor is a factor that is regular of degree $1$.

The first case to be considered is  $n=4k+2$.
 Let $A_i=\left\{i+1, i+2k+2\right\}$ and $B_i=\left\{2i+1,2i+2\right\}$, $0\le i\le 2k$. Then $\mathcal{A}=\left\{A_i: 0\le i\le 2k\right\}$ and $\mathcal{B}=\left\{B_i: 0\le i\le 2k\right\}$ are two  $1$-factors of $K_n$.

\begin{lemma}\label{LA}
 Let $1\leq m\le 2k$, $0\le j_1<j_2<...<j_m\le 2k$, $0\le i_1<i_2<...<i_w\le 2k$. If $\bigcup\limits_{u=1}^mB_{j_u}\subseteq \bigcup\limits_{v=1}^wA_{i_v}$ holds, then $w\ge m+1$.
\end{lemma}

\noindent{\it Proof:} Let $A=\bigcup\limits_{v=1}^wA_{i_v}$ and $B=\bigcup\limits_{u=1}^mB_{j_u}$. It is easy to see that $|A|=2w$ and $|B|=2m$ since $\mathcal{A}$ and $\mathcal{B}$ are all $1$-factors of $K_n$. So we have $w\ge m$ since $B\subseteq A$. Thus we only need to prove $w\not =m$. Seeking a contradiction, suppose that $w=m$ and $A=B$.  We distinguish two cases.

(1) $m$ is even. Since $i_1+1$, $2j_1+1$ are the smallest elements of $A$ and $B$, respectively, then $i_1+1=2j_1+1$, i.e. $i_1=2j_1$. Next, $i_2+1$ and $2j_1+2$ are the smallest elements of $P^1=A\setminus\{i_1+1\}$ and $Q^1=B\setminus\{2j_1+1\}$, respectively, so $i_2=2j_1+1$, thus $i_2=i_1+1$. Furthermore, $i_3+1$ and $2j_2+1$ are the smallest elements of $P^2=P^1\setminus\{i_2+1\}$ and $Q^2=Q^1\setminus\{2j_1+2\}$, respectively, so $i_3=2j_2$. And so on, we have $i_{2k-1}=2j_k$ and $i_{2k}=i_{2k-1}+1$, $k=1,2,...,\frac{m}{2}$.
 Now, the smallest element of $Q^{m}=B\setminus\{2j_1+1,2j_1+2,2j_2+1,2j_2+2,...,2j_{\frac{m}{2}}+1,2j_{\frac{m}{2}}+2\}$ is $2j_{{\frac{m+2}{2}}}+1$ which is odd, however, the smallest element in $P^{m}=A\setminus\{i_1+1,i_2+1,...,i_m+1\}$ is $i_1+2k+2=2j_1+2k+2$ which is even, a contradiction.

(2) $m$ is odd.  Since $w=m$ and $A=B$, then we have $\bar{A} = [4k+2]\setminus A=  [4k+2]\setminus B= \bar{B}$ and both of these consist the union of $(2k+1)-m$ sets.
Thus, if $m$ is odd, then  $(2k+1)-m$ is even. We recover a contradiction in the same way as in case (1).\qed

Now we give the construction of a family of even-sized subsets with $|\mathcal{A}_s|= 2^{\lfloor \frac{n}{2} \rfloor}+s$. Let $\mathcal{A}_0=\{\bigcup\limits_{T_i\in T}T_i:T\in 2^{\mathcal{A}}\}$ and $\mathcal{B}_0=\{\bigcup\limits_{T_i\in T} T_i: T\in 2^\mathcal{B},\ T\neq \emptyset,\ T\neq \mathcal{B}\}$.  Let $\mathcal{A}_s=\mathcal{A}_0\cup \mathcal{D}$, where  $\mathcal{D}\subseteq \mathcal{B}_0$ and $|\mathcal{D} |=s$, $1\leq s \leq 2^{\lfloor\frac{n}{2}\rfloor}-2$.

\begin{theorem}\label{ThA}
Let $n\equiv 2 \pmod 4$. There exists  a family of even-sized subsets $\mathcal{A}_s$ with $|\mathcal{A}_s|= 2^{\lfloor \frac{n}{2} \rfloor}+s$ and $op(\mathcal{A}_s)= s\cdot 2^{\lfloor \frac{n}{2} \rfloor-1} $.
\end{theorem}

\noindent{\it Proof:}
It follows that all subsets of $[n]$ in $\mathcal{A}_0$ and $\mathcal{B}_0$ are even-sized since both $\mathcal{A}$ and $\mathcal{B}$ are 1-factors. So $\mathcal{A}_s$ is  a family of even-sized subsets.
By Lemma \ref{LA}, we know that $\mathcal{A}_0\cap \mathcal{D}=\emptyset$. Hence, the size of $\mathcal{A}_s$ is $2^{\mid\mathcal{A}\mid}+s=2^{2k+1}+s=2^{\lfloor\frac{n}{2}\rfloor}+s$.

Next, we prove that the odd pair number of $\mathcal{A}_s$ is exactly $s\cdot 2^{2k}$(i.e. $s\cdot 2^{\lfloor \frac{n}{2} \rfloor-1}$). It's easy to check that no two elements of $\mathcal{A}_0$ or $\mathcal{D}$ can have odd-sized intersections by our construction. Thus it suffices to count the odd pair number between $\mathcal{A}_0$ and $\mathcal{D}$.

Now we show that for any $B\in \mathcal{B}_0$,   $op(\mathcal{A}_0\cup \{B\})=2^{2k}$. Suppose  $B$ is the union of  $m$ ($1\leq m\leq 2k$) pairs of $\mathcal{B}$, i.e. $B=\bigcup\limits_{u=1}^mB_{j_u}$.
Let $T=\bigcup\limits_{v=1}^wA_{i_v}$ be the unique smallest set in $\mathcal{A}_0$ such that $B\subseteq T$. By Lemma \ref{LA} we have $m+1\leq w\leq 2m$. Thus there are exactly $2w-2m$ elements, say $x_1,x_2,...,x_{2w-2m}$, in $E=T\setminus B$, and $|A_l\cap E|\le 1$ for any $0\le l\le 2k$. W.l.g, we may assume $x_h\in A_{i_h}$.
Further, let $F=\bigcup\limits_{h=1}^{2w-2m}(A_{i_h}\setminus \{x_h\})$. Then for any $A\in \mathcal{A}_0$,  $|A\cap B|$ is odd if and only if $|A\cap F|$ is odd.
So we have
\begin{equation*}
\begin{split}
op(\mathcal{A}_0\cup \{B\})&= \sum_{\substack{1\leq i\leq 2w-2m\\ i\ is\ odd}} \dbinom{2w-2m}{i}\cdot 2^{2k+1-(2w-2m)}\\
                 &= 2^{2k+1-(2w-2m)} \sum_{\substack{1\leq i\leq 2w-2m\\ i\ is\ odd}} \dbinom{2w-2m}{i}\\
                 &= 2^{2k+1-(2w-2m)}\cdot 2^{2w-2m-1}\\
                 &=2^{2k}.
\end{split}
\end{equation*}
Thus, we have  $op(\mathcal{A}_s)=s\cdot 2^{2k}$.\qed

Similarly, we can give the construction when $n=4k$.
Let $A_i=\{i+1,2k+i+1\}, 0\leq i\leq 2k-1$,
$B_0=\{2k+1, 4k\}$, $B_j=\{2j+1,2j+2\}$, $0\leq j\leq k-1$, and $B_{j}^{'}=\{2j+2k,2j+2k+1\}$, $1\leq j \leq k-1$. Denote $\mathcal{B}_1=\{B_j: 1\leq j\leq k\}$ and $\mathcal{B}_2=\{B_{j}^{'}:1\leq j\leq k-1\}$. Then  $\mathcal{A}=\{A_i: 0\leq i\leq 2k-1\}$ and $\mathcal{B}=B_0\cup \mathcal{B}_1\cup \mathcal{B}_2$ are two $1$-factors of $K_n$.

Let $\E$ and $\F$ be subsets of $\mathcal{B}_1$ and $\mathcal{B}_2$, respectively, with $|\E|=m_1$ and $|\F|=m_2$.
Let $E=\bigcup\limits_{B_j\in \E}B_j$ and $F=\bigcup\limits_{B'_{j}\in \F}B'_{j}$.
It's certain when $m_1=0$ and $E \cup F \subseteq \bigcup\limits_{v=1}^wA_{i_v}$ holds, there is $w\ge m_2+1$. So as $m_2=0$, there is $w\ge m_1+1$.

\begin{lemma}\label{TLA}
Let $1\leq m\leq 2k-1$, $0\le j_1<j_2<...<j_{m_1}\le k-1$, $1\leq j'_1<j'_2<...<j'_{m_2} \le k-1$ and $0\le i_1<i_2<...<i_w\le 2k-1$. Denote $E=\bigcup\limits_{u=1}^{m_1}B_{j_u}$, $F=\bigcup\limits_{u=1}^{m_2}B^{'}_{j'_u}$, we have

$(1)$ If $m_1+m_2=m$ and $E \cup F \subseteq \bigcup\limits_{v=1}^wA_{i_v}$ holds, then $w\ge m+1$.

$(2)$ If $m_1+m_2=m-1$ and $B_0\cup E \cup F \subseteq \bigcup\limits_{v=1}^wA_{i_v}$ holds, then $w\ge m+1$.
\end{lemma}

\noindent{\it Proof:}
%\begin{description}
(1) Let $A=\bigcup\limits_{v=1}^wA_{i_v}$ and $B=E \cup F$. Observe that the elements of $E$ are all smaller than the elements of $F$. It is easy to see that $|A|=2w$ and $|B|=|E\cup F|=2(m_1+m_2)=2m$, since $\mathcal{A}$ and $\mathcal{B}$ are all $1$-factors of $K_n$. So we have $w\ge m$ since $B\subseteq A$. Thus we only need to prove $w\not =m$. Otherwise, $w=m$ and $A=B$. Here, we distinguish it into two cases.
%\begin{description}

$\mathbf{Case~ 1}$ $m$ is even.
%\begin{itemize}

$\mathbf{Case~ 1.1}$ $m_1<\frac{m}{2}$.

 Since $i_1+1$, $2j_1+1$ are the smallest elements of $A$ and $B$, respectively, then $i_1+1=2j_1+1$, i.e. $i_1=2j_1$. Next, $i_2+1$ and $2j_1+2$ are the smallest elements of $P^1=A\setminus\{i_1+1\}$ and $Q^1=B\setminus\{2j_1+1\}$, respectively, so $i_2=2j_1+1$, thus $i_2=i_1+1$. Furthermore, $i_3+1$ and $2j_2+1$ are the smallest elements of $P^2=P^1\setminus\{i_2+1\}$ and $Q^2=Q^1\setminus\{2j_1+2\}$, respectively, so $i_3=2j_2$. And so on, we have $i_{2l-1}=2j_l$ and $i_{2l}=i_{2l-1}+1$, $l=1,2,...,m_1$.

 Now, the smallest elements of $Q^{2m_1}=B\setminus\{2j_1+1,2j_1+2,2j_2+1,2j_2+2,...,2j_{m_1}+1,2j_{m_1}+2\}$ and $P^{2m_1}=A\setminus\{i_1+1,i_2+1,...,i_{2m_1-1}+1,i_{2m_1}+1\}$ are $2j'_1+2k$ and $i_{2m_1+1}+1$. So, $i_{2m_1+1}+1=2j'_1+2k$.
 Since $i_{2m_1+1}\le 2k-1$, then $i_{2m_1+1}+1\le 2k<2j'_1+2k$ , a contradiction.

$\mathbf{Case~ 1.2}$ $m_1=\frac{m}{2}$.

 Similar to case 1, there is  $i_{2l-1}=2j_l$ and $i_{2l}=i_{2l-1}+1$, $l=1,2,...,\frac{m}{2}$. Now, the smallest element in $B\setminus\{2j_1+1,2j_1+2,2j_2+1,2j_2+2,...,2j_{\frac{m}{2}}+1,2j_{\frac{m}{2}}+2\}=B\setminus E=F$ is $2j'_1+2k$ which is even. However, the smallest element of $A\setminus\{i_1+1,i_2+1,...,i_m+1\}$ is $i_1+2k+1=2j_1+2k-1$ which is odd, a contradiction.

$\mathbf{Case~ 1.3}$ $m_1>\frac{m}{2}$.

By Case 1.2, there is  $i_{2l-1}=2j_l$ and $i_{2l}=i_{2l-1}+1$, $l=1,2,...,\frac{m}{2}$. Now, the smallest element in $Q^{m}=B\setminus\{2j_1+1,2j_1+2,2j_2+1,2j_2+2,...,2j_{\frac{m}{2}}+1,2j_{\frac{m}{2}}+2\}$ is $2j_{\frac{m}{2}+1}+1$. And the smallest element of $P^{m}=A\setminus\{i_1+1,i_2+1,...,i_m+1\}$ is $i_1+2k+1$, so $ i_1+2k+1=2j_{\frac{m}{2}+1}+1$. However, $i_1+2k+1>2k$, $2j_{\frac{m}{2}+1}+1\le 2k$ , a contradiction.

% \end{itemize}

$\mathbf{Case~ 2}$ $m$ is odd.

$\mathbf{Case~ 2.1}$ $m_1<\frac{m-1}{2}$.

Similar to the case 1.1.

$\mathbf{Case~ 2.2}$ $m_1=\frac{m-1}{2}$.

Similar to the case 2.2, we have $i_{2l-1}=2j_l$ and $i_{2k}=i_{2k-1}+1$, $l=1,2,...,\frac{m-1}{2}$. Now, the smallest element in $B\setminus\{2j_1+1,2j_1+2,2j_2+1,2j_2+2,...,2j_{\frac{m-1}{2}}+1,2j_{\frac{m-1}{2}}+2\}$ is $2j'_1+2k$. And the smallest element of $A\setminus\{i_1+1,i_2+1,...,i_{m-1}+1\}$ is $i_m+1$, so $i_m+1=2j'_1+2k$ which implies $i_m-2k=2j'_1-1>0$. Since $i_m\leq 2k-1$, we have  a contradiction.

$\mathbf{Case~ 2.3}$ $m_1>\frac{m-1}{2}$.

Similar to the case 1.3, we have $i_{2k-1}=2j_k-2$ and $i_{2l}=i_{2l-1}$, $l=1,2,...,\frac{m-1}{2}$.
Now, the smallest elements in $Q^{m-1}=B\setminus\{2j_1+1,2j_1+2, 2j_2+1,2j_2+2,...,2j_{\frac{m-1}{2}}+1,2j_{\frac{m-1}{2}}+2\}$ and $P^{m-1}=A\setminus\{i_1+1,i_2+1,...,i_{m-1}+1\}$ are $2j_{\frac{m+1}{2}}+1$ and $i_m+1$, so $i_m=2j_{\frac{m+1}{2}}$.
Next, $2j_{\frac{m+1}{2}}+2$ and $i_1+2k+1$ are the smallest elements of $Q^{m-1}\setminus\{2j_{\frac{m+1}{2}}+1\}$ and $P^{m-1}\setminus\{i_m+1\}$. Then, we have $i_1+2k+1=2j_{\frac{m+1}{2}}+2$ is even. However, by $i_1=2j_1$, we have that $i_1+2k+1=2j_1+2k+1$ which is odd, a contradiction.

(2) Let $A=\bigcup\limits_{v=1}^wA_{i_v}$ and $B=B_0\cup E\cup F$. At first, if $m=1$, the condition equals to $B_0\subseteq A$. Besides, it's easy to see the elements of $B_0=\{2k,4k\}$ has already appeared in $A_0\cup A_{2k-1}$, which means $w=2\geq 1+1$. The statement is true.
On the other hand, when $m>1$, since $B_0\subseteq A$, there is $A_0, A_{2k-1}\subseteq A$. The condition equals to $E\cup F \subseteq A\setminus\{A_0,A_{2k-1}\}$ with $m_1+m_2=m-1$. Moreover, we have $w-2\geq m-1$ by $(1)$, thus $w\geq m+1$.
\qed

Let  $\mathcal{A}_0=\{\bigcup\limits_{T_i\in T}T_i:T\in 2^{\mathcal{A}}\}$ and $\mathcal{B}_0=\{\bigcup\limits_{T_i\in T} T_i: T\in 2^\mathcal{B},\ T\neq \emptyset,\ T\neq \mathcal{B}\}$. Then, we can construct a family of even-sized subsets $\mathcal{A}_s=\mathcal{A}_0\cup \mathcal{D}$, where  $\mathcal{D}\subseteq \mathcal{B}_0$ and $|\mathcal{D} |=s$, $1\leq s \leq 2^{\lfloor\frac{n}{2}\rfloor}-2$.

\begin{theorem}\label{ThB}
Let $n\equiv0 \pmod 4$. There exists  a family of even-sized subsets $\mathcal{A}_s$ with $|\mathcal{A}_s|= 2^{\lfloor \frac{n}{2} \rfloor}+s$ and $op(\mathcal{A}_s)= s\cdot 2^{\lfloor \frac{n}{2} \rfloor-1} $.
\end{theorem}

\noindent{\it Proof:} It follows that all subsets of $[n]$ in $\mathcal{A}_0$ and $\mathcal{B}_0$ are even-sized since both $\mathcal{A}$ and $\mathcal{B}$ are 1-factors. So $\mathcal{A}_s$ is  a family of even-sized subsets.
By Lemma \ref{TLA}, we know that $\mathcal{A}_0\cap \mathcal{D}=\emptyset$. Hence, the size of $\mathcal{A}_s$ is $2^{\mid\mathcal{A}\mid}+s=2^{2k}+s=2^{\lfloor\frac{n}{2}\rfloor}+s$.

Next, we prove that the odd pair number of $\mathcal{A}_s$ is exactly $s\cdot 2^{2k-1}$(i.e. $s\cdot 2^{\lfloor \frac{n}{2} \rfloor-1}$). It's easy to check that no two elements of $\mathcal{A}_0$ or $\mathcal{D}$ can have odd-sized intersections by our construction. Thus it suffices to count the odd pair number between $\mathcal{A}_0$ and $\mathcal{D}$.

At first, we will show that for any $B\in \mathcal{B}_0$, there is $op(\mathcal{A}_0\cup \{B\})=2^{2k-1}$. W.l.g, we may suppose  $B$ is the union of  $m$ ($1\leq m\leq 2k$) pairs of $\mathcal{B}$. Let $T=\bigcup\limits_{v=1}^wA_{i_v}$ be the unique smallest set in $\mathcal{A}_0$ such that $B\subseteq T$.

(1) If $B_0\subset B$, we may assume $B=B_0\cup (\bigcup\limits_{u=1}^{m_1}B_{j_u})\cup (\bigcup\limits_{u=1}^{m_2}B_{j'_u})$, where $m_1+m_2=m-1$.
By Lemma \ref{TLA}, we have  $m+1\leq w\leq 2m$.
Thus there are $2w-2m$ elements, say $x_1,x_2,.., x_{2w-2m}$, in $E=T\setminus B$, and $|A_l\cap E|\le 1$ for any $0\le l\le 2k-1$. W.l.g, we may assume $x_h\in A_{i_h}$.
Further, let $F=\bigcup\limits_{h=1}^{2w-2m}(A_{i_h}\setminus \{x_h\})$. Then for any $A\in \mathcal{A}_0$,  $|A\cap B|$ is odd if and only if $|A\cap F|$ is odd.
So we have
\begin{equation*}
\begin{split}
op(\mathcal{A}_0\cup \{B\})&= \sum_{\substack{1\leq i\leq 2w-2m\\ i\ is\ odd}} \dbinom{2x-2m}{i}\cdot 2^{2k-(2w-2m)}\\
                 &= 2^{2k-(2w-2m)}\cdot 2^{2w-2m-1}\\
                 &=2^{2k-1}.
\end{split}
\end{equation*}
(2) If $B_0\nsubseteq B$, we may assume $B=(\bigcup\limits_{u=1}^{m_1}B_{j_u})\cup (\bigcup\limits_{u=1}^{m_2}B_{j'_u})$, where $m_1+m_2=m$. Let $A=\bigcup\limits_{v=1}^wA{i_v}$, and when $B\subset A$, there is $m+1\leq w\leq 2m$ by Lemma \ref{TLA}. Similar to $(1)$, we have
$$op(\mathcal{A}_0\cup \{B\})= \sum_{\substack{1\leq i\leq 2w-2m\\ i\ is\ odd}} \dbinom{2x-2m}{i}\cdot 2^{2k-(2w-2m)}=2^{2k-1}.$$
Finally, we have $op(\mathcal{A}_s)=s\cdot 2^{2k-1}$.\qed

\noindent{\it \underline{Proof of Theorem \ref{Th}:}}
When $n\equiv 1 \pmod 4$, suppose $n=4k+1$ and we can get 1-factor of $K_{4k}$ by deleting the last element $4k+1$. By Theorem \ref{ThB} there exists a family of even-sized subsets $\mathcal{A}_s$ with $|\mathcal{A}_s|=2^{\lfloor \frac{4k+1}{2} \rfloor}+s$ and $op(\mathcal{A}_s)= s\cdot 2^{\lfloor \frac{4k+1}{2} \rfloor-1} $. On the other hand, for the case $n\equiv 3 \pmod 4$, let $n=4k+3$ and we can delete the last element $4k+3$ to get a $1$-factor of $K_{4k+2}$. Hence, there exists a family of even-sized subsets $\mathcal{A}_s$ with $|\mathcal{A}_s|=2^{\lfloor \frac{4k+3}{2} \rfloor}+s$ and $op(\mathcal{A}_s)= s\cdot 2^{\lfloor \frac{4k+3}{2} \rfloor-1}$ by Theorem \ref{ThA}. Combining those above, Theorem \ref{Th} is proved.\qed

\section{Eventown via symmetric design}

In this section we give a construction using symmetric designs. For definitions in combinatorial designs, one can refer to \cite{MR0002086} and \cite{MR0006834}. Here, we introduce some useful notations and then present the example.

\begin{definition}
A balanced incomplete block design (BIBD) is a pair $(V, \mathcal{B})$ where $V$ is a $v$-set(point set) and $\mathcal{B}$ is a collection of $b$ $k$-subsets of  $V$ (blocks) such that  any 2-subset of $V$ is contained in exactly $\lambda$ blocks.
Furthermore, if $b=v$ it is called symmetric and denoted by $(v, k, \lambda)$-SBIBD.
\end{definition}

\begin{lemma}\label{SP}(\cite{MR2246267})
In a $(v, k, \lambda)$-SBIBD, every two distinct blocks have $\lambda$ points in common.
\end{lemma}

 \begin{theorem}\label{PM}(\cite{MR2246267})
If $4k-1$ is a prime power, then a $(4k-1, 2k, k)$-SBIBD exists.
\end{theorem}

First, we list an example constructed by a $(7,4,2)$-SBIBD. Let $A_i=\{2i-1,2i\}$ and $B_i=\{ i, i+1, i+4, i+6\}\pmod 7$, then $\mathcal{A}=\{A_i: 1\leq i\leq3\}$ is a 1-factor of $K_6$ and $(\mathbb{Z}_7,\mathcal{B})$ is a  $(7,4,2)$-SBIBD, where $\mathcal{B}=\{B_i: i\in \mathbb{Z}_7\}$. Now, we can give the construction of a family of even-sized subsets $\mathcal{A}_s=\mathcal{A}_0\cup \mathcal{D}$ with $|\mathcal{A}_s|= 2^{\lfloor \frac{7}{2} \rfloor}+s$, where $\mathcal{A}_0=\{\bigcup\limits_{T_i\in T}T_i:T\in 2^{\mathcal{A}}\}$,  $\mathcal{D}\subseteq \mathcal{B}$ and $|\mathcal{D}|=s$, $1\leq s \leq 7$. It's easy to check that no two elements of $\mathcal{A}_0$ or $\mathcal{D}$ can have odd-sized intersections by our construction. Thus it suffices to count the odd pair number between $\mathcal{A}_0$ and $\mathcal{D}$.
Moreover, there is $op(\mathcal{A}_0\cup B_i)=4$ for any block $B_i$, $i\in \mathbb{Z}_7$. Therefore, the odd pair number of $\mathcal{A}_s$ is $s\cdot 2^{{\lfloor \frac{7}{2} \rfloor}-1}$. Observe that when $s=7$, we get a family of even-sized subsets with size $8+7=15$ and odd pair number is $7\cdot 2^{{\lfloor \frac{7}{2} \rfloor}-1}$=$28$.

Inspired by this example, we attempt to generalize it by the $(4k-1, 2k, k)$-SBIBD when $k$ is even. Finally, we get a family of even-sized subsets $\mathcal{A}_s$, with $|\mathcal{A}_s|=2^{\lfloor\frac{4k-1}{2}\rfloor}+s$ where $1\leq s\leq4k-1$. Here, we state that below.

\begin{construction}\label{Sc}
Let $A_i=\{2i-1, 2i\}$, $1\leq i\leq 2k-1$, and $\mathcal{A}=\{A_i: 1\leq i\leq 2k-1\}$ be a 1-factor of $K_{4k-2}$. Let $\mathcal{A}_0=\{\bigcup\limits_{T_i\in T}T_i:T\in 2^{\mathcal{A}}\}$ with $|\mathcal{A}_0|=2^{\lfloor\frac{4k-1}{2}\rfloor}$.  Let $\mathcal{B}$ be the block set of a symmetric $(4k-1, 2k, k)$-design with the point set $[4k-1]$. Then, we get a family of even-sized subsets $\mathcal{A}_s=\mathcal{A}_0 \cup  \mathcal{D}$ with $|\mathcal{A}_s|= 2^{\lfloor \frac{4k-1}{2} \rfloor}+s$, where $\mathcal{D}\subseteq \mathcal{B}$ and $|\mathcal{D}|=s$, $1\leq s \leq 4k-1$.

\end{construction}

\begin{lemma}\label{Sdl}
Let $k$ be even. $op(\mathcal{A}_0\cup B_i)=2^{{\lfloor \frac{4k-1}{2} \rfloor}-1}$, where $\mathcal{A}_0$ and $B_i$ are construed in Construction~\ref{Sc}.
\end{lemma}
 \noindent{\it Proof:}
  Let $1\leq w\leq 2k-1$. For any $B_i\in \mathcal{B}$, let $T=\bigcup\limits_{v=1}^wA_{i_v}$ be the unique smallest set in $\mathcal{A}_0$ such that $B_i\subseteq T$. Since $|B_i\cap A_i|\leq 2$ and $|B_i|=2k$, if $4k-1\notin B_i$ and $B_i\subseteq T$, then  $w\ge k$. If $4k-1\in B_i$, then $B_i\subseteq T$ is equivalent to  $B_i\setminus\{4k-1\} \subseteq T$, which also implies $w\ge k$. So we have $k\leq w\leq 2k-1$. Thus there are exactly $2w-2k$ elements, say $x_1,x_2,...,x_{2w-2k}$, in $E=T\setminus B_i$, and $|A_l\cap E|\le 1$ for any $1\le l\le 2k-1$. W.l.g, we may assume $x_h\in A_{i_h}$.

Further, let $F=\bigcup\limits_{h=1}^{2w-2k}(A_{i_h}\setminus \{x_h\})$. Then for any $A\in \mathcal{A}_0$,  $|A\cap B_i|$ is odd if and only if $|F\cap B_i|$ is odd.
So we have
\begin{equation*}
\begin{split}
op(\mathcal{A}_0\cup \{B_i\})&= \sum_{\substack{1\leq i\leq 2w-2k\\ i\ is\ odd}} \dbinom{2w-2k}{i}\cdot 2^{2k-1-(2w-2k)}\\
                 &= 2^{2k-1-(2w-2k)} \sum_{\substack{1\leq i\leq 2w-2k\\ i\ is\ odd}} \dbinom{2w-2k}{i}\\
                 &= 2^{2k-1-(2w-2k)}\cdot 2^{2w-2k-1}\\
                 &=2^{{\lfloor \frac{4k-1}{2} \rfloor}-1}.
\end{split}
\end{equation*}
\qed

\noindent{\it \underline{Proof of Theorem \ref{SD}:}} Combining Construction~\ref{Sc}, Theorem~\ref{PM}, and Lemmas~\ref{SP} and~\ref{Sdl}, we have proven Theorem~\ref{SD}.\qed

\noindent{\bf Acknowledgments}\
This work was supported by the National Natural Science Foundation of China under Grant
Nos.12471313, 12071226 (Haitao Cao), and 12101268 (Xiaolei Niu).

\bibliographystyle{plain}

\begin{thebibliography}{10}

\bibitem{MR0249303}
E. R. Berlekamp.
\newblock On subsets with intersections of even cardinality.
\newblock {\em Canad. Math. Bull.}, 12(4):471--474, 1969.

\bibitem{MR2246267}
C. J. Colbourn and J. H. Dinitz, editors.
\newblock {\em Handbook of combinatorial designs}.
\newblock Discrete Mathematics and its Applications (Boca Raton). Chapman \&
  Hall/CRC, Boca Raton, FL, second edition, 2007.

\bibitem{MR0729786}
M. Deza, P. Frankl, and N. M. Singhi.
\newblock On functions of strength {$t$}.
\newblock {\em Combinatorica}, 3(3-4):331--339, 1983.

\bibitem{MR0140419}
P.~Erd\H{o}s, Chao Ko, and R.~Rado.
\newblock Intersection theorems for systems of finite sets.
\newblock {\em Quart. J. Math. Oxford Ser. (2)}, 12:313--320, 1961.

\bibitem{MR0002086}
R. A. Fisher.
\newblock An examination of the different possible solutions of a problem in incomplete blocks.
\newblock {\em Ann. Eugenics}, 10:52--75, 1940.

\bibitem{MR0725068}
P. Frankl and A. M. Odlyzko.
\newblock On subsets with cardinalities of intersections divisible by a fixed integer.
\newblock {\em Europ. J. Combin.}, 4(3):215--220, 1983.

\bibitem{MR3656338}
P. Frankl.
\newblock A stability result for families with fixed diameter.
\newblock {\em Combin. Probab. Comput.}, 26(4):506--516, 2017.

\bibitem{MR3838458}
D. Gerbner and B.~Patk\'{o}s.
\newblock {\em Extremal finite set theory}.
\newblock Discrete Mathematics and its Applications (Boca Raton). CRC Press,
  Boca Raton, FL, 2019.

\bibitem{MR0376400}
J. Graver.
\newblock Boolean designs and self-dual matroids.
\newblock {\em Linear Algebra Appl.}, 10:111--128, 1975.

\bibitem{JO2023}
 G. Johnston and J. O'Neill.  A few new oddtown and eventown problems. arXiv preprint arXiv:2312.13588.

\bibitem{MR0200179}
D. J. Kleitman.
\newblock On a combinatorial conjecture of {E}rd\H{o}s.
\newblock {\em J. Combinatorial Theory}, 1:209--214, 1966.

\bibitem{MR0006834}
F. W. Levi.
\newblock {\em Finite {G}eometrical {S}ystems}.
\newblock University of Calcutta, Calcutta, 1942.

\bibitem{MR4564468}
J. O'Neill.
\newblock A short note on supersaturation for oddtown and eventown.
\newblock {\em Discret. Appl. Math.}, 334:63--67, 2023.

\bibitem{MR4429337}
J. O'Neill and J.~Verstra\"{e}te.
\newblock A note on {$k$}-wise oddtown problems.
\newblock {\em Graphs Combin.}, 38(3):101, 2022.

\bibitem{MR1343352}
H. S. Snevily.
\newblock A generalization of the {R}ay-{C}haudhuri-{W}ilson theorem.
\newblock {\em J. Combin. Des.}, 3(5):349--352, 1995.

\bibitem{MR1544925}
E. Sperner.
\newblock Ein {S}atz\"{u}ber {U}ntermengen einer endlichen {M}enge.
\newblock {\em Math. Z.}, 27(1):544--548, 1928.

\bibitem{MR3755659}
B. Sudakov and P. Vieira.
\newblock Two remarks on eventown and oddtown problems.
\newblock {\em SIAM J. Discrete Math.}, 32(1):280--295, 2018.

\bibitem{MR2191021}
T. Szab\'{o} and V. Vu.
\newblock Exact {$k$}-wise intersection theorems.
\newblock {\em Graphs Combin.}, 21(2):247--261, 2005.

\bibitem{MR1451492}
V. Vu.
\newblock Extremal systems with upper-bounded odd intersections.
\newblock {\em Graphs Combin.}, 13(2):197--208, 1997.

\bibitem{MR1773658}
V. Vu.
\newblock Extremal set systems with weakly restricted intersections.
\newblock {\em Combinatorica}, 19(4):567--588, 1999.

\bibitem{Wei2023OnSF}
X. Wei, Y. Zhao, X. Zhang, and G. Ge.
\newblock On supersaturation for oddtown and eventown.
\newblock arXiv preprint arXiv:2302.05586.


\bibitem{Wei2025DCC}
X. Wei, X. Zhang, and G. Ge.
\newblock On set systems with strongly restricted intersections.
\newblock {\em Des. Codes  Cryptogr. } 93:667--682, 2025.

\end{thebibliography}

\end{sloppypar}
\end{document}